\documentclass[11pt]{amsart}
\usepackage{amssymb,latexsym}
\input xypic
  \xyoption{all}

\usepackage{amsmath, amsthm, amsfonts}

\usepackage{pifont,pxfonts,txfonts}
\usepackage{yfonts}
\usepackage{bbm}
\usepackage{calrsfs}
\usepackage{empheq}
\usepackage{color}
\usepackage[colorlinks,linkcolor=red,anchorcolor=blue,citecolor=green]{hyperref}
\usepackage{amsmath, amstext, amsbsy, amscd}
\usepackage[mathscr]{eucal}
\usepackage{times}
\usepackage{amsmath, amsthm, amsfonts,extarrows}

\newtheorem{theorem}{Theorem}[section]

\newtheorem{lemma}[theorem]{Lemma}
\newtheorem{remark}[theorem]{Remark}

   \newcommand{\ba}{\begin{eqnarray}}
   \newcommand{\na}{\end{eqnarray}}
   \newcommand{\ban}{\begin{eqnarray*}}
   \newcommand{\nan}{\end{eqnarray*}}

\newcommand{\bA}{{\mathbb A}}
\newcommand{\bC}{{\mathbb C}}
\newcommand{\bF}{{\mathbb F}}
\newcommand{\bP}{{\mathbb P}}
\newcommand{\bQ}{{\mathbb Q}}
\newcommand{\bZ}{{\mathbb Z}}
\newcommand{\cO}{{\mathcal O}}
\newcommand{\fz}{{\mathfrak z}}

  \newcommand{\<}{\langle}
  \renewcommand{\>}{\rangle}

\begin{document}

\title{Stable pair invariants under blow-ups}

\author[Hua-Zhong Ke]{Hua-Zhong Ke}
  \address{Department of Mathematics\\  Sun Yat-Sen University\\
                        Guangzhou,  510275\\ China }
  \email{kehuazh@mail.sysu.edu.cn}

\begin{abstract}

We use degeneration formula to study the change of stable pair invariants of $3$-folds under blow-ups and obtain some closed blow-up formulae. Related results on Donaldson-Thomas invariants are also discussed. Our results give positive evidence for GW/DT/P correspondence, and also give partial correspondence for varieties not necessarily toric or complete intersections.

\end{abstract}

\date{\today}
\maketitle
\tableofcontents

\section{Introduction}

Curve counting theories have played prominent roles in both mathematics and physics in the last two decades. For any nonsingular $3$-fold $X$, there are (at least) three different curve counting theories on $X$. Much studied Gromov-Witten theory counts stable maps from curves to $X$. Donaldson-Thomas theory \cite{DT, Th} counts one dimensional subschemes in $X$. Stable pair theory, introduced by Pandharipande and Thomas in \cite{PT}, counts pairs $(C, D)$ where $C\subset X$ is an embedded curve and $D$ is a divisor on $C$. It is conjectured that these three curve counting theories on $X$ are equivalent \cite{MNOP1, MNOP2, PT, PP4}. GW/DT/P correspondence has been proved in many important cases, including quintic $3$-folds \cite{Br, MOOP, OP, PP4, PP5, T1}. This suggests that many phenomena in one theory have counterparts in the other two theories.

A fundamental problem in Gromov-Witten theory is to understand how Gromov-Witten invariants change under surgeries \cite{LR, R}. For $3$-folds, the first breakthrough in this direction is the work of Li and Ruan \cite{LR} on the transformation of Gromov-Witten invariants under flops and extremal small transitions. In birational geometry, blow-up is an elementary surgery, but it is rare to be able to obtain closed blow-up formulae for Gromov-Witten invariants. In the last twenty years, only a few limited cases were known \cite{Ga, H1, H2, HHKQ, HLR}. It is also important to study the effect of surgeries on Donaldson-Thomas theory. Hu and Li \cite{HL} have studied the transformation of Donaldson-Thomas invariants under blow-ups at points, ordinary flops and extremal small transitions. For general flops between Calabi-Yau $3$-folds, Toda \cite{T2} has established the flop formula for Donaldson-Thomas invariants, and hence for stable pair invariants due to the DT/P correspondence in the Calabi-Yau case \cite{Br, T1}. In this paper, we study the transformation of stable pair invariants under blow-ups.

Throughout this paper, let $X$ be an irreducible, nonsingular, projective $3$-fold over $\bC$, and $p:\tilde{X}\rightarrow X$ the blow-up of $X$ at a point $P$ or along an irreducible, nonsingular embedded curve $C$ of $X$. Let $E$ be the exceptional divisor of the blow-up, and $e\in H_2(\tilde X,\bZ)$ the class of a line in the fiber of $E$. Note that $p$ induces a natural injection via 'pull-back' of $2$-cycles 
\ban
p^!=PD_{\tilde X}\circ p^*\circ PD_X:H_2(X,\bZ)\rightarrow H_2(\tilde X,\bZ),
\nan
where the image of $p^!$ is the subset of $H_2(\tilde X,\bZ)$ consisting of $2$-cycles having intersection number zero with $E$. We will compare partition functions $Z_P$ of stable pair invairants of $X$ and those of $\tilde X$, the definition of which will be reviewed in Section 2.

We first consider blow-up at a point.
\begin{theorem}\label{pt0}
Let $p:\tilde X\rightarrow X$ be the blow-up at a point. Suppose  that $\gamma_1,\cdots,\gamma_m\in H^{>0}(X, {\mathbb Q})$, and $d_1,\cdots,d_m\in\bZ_{\geqslant 0}$. Then for any $\beta\in H_2(X, {\mathbb Z})$ and $k\in\bZ_{>0}$, we have
\ban
Z_P\bigg(\tilde X;q|\prod\limits_{i=1}^m\tau_{d_i}(p^*\gamma_i)\bigg)_{p^!\beta+ke}=0.
\nan
\end{theorem}
\begin{theorem}\label{pt1}
Let $p:\tilde X\rightarrow X$ be the blow-up at a point. Suppose  that $\gamma_1,\cdots,\gamma_m\in H^{>0}(X, {\mathbb Q})$, and $d_1,\cdots,d_m\in\bZ_{\geqslant 0}$. Then for any nonzero $\beta\in H_2(X, {\mathbb Z})$, we have
\ban
Z_P\bigg(X;q|\prod\limits_{i=1}^m\tau_{d_i}(\gamma_i)\bigg)_\beta=Z_P\bigg(\tilde X;q|\prod\limits_{i=1}^m\tau_{d_i}(p^*\gamma_i)\bigg)_{p^!\beta}.
\nan
\end{theorem}
\begin{theorem}\label{pt2}
Under the same assumptions as in Theorem \ref{pt1}, we have
\ban
Z_P\bigg(X;q|\tau_0([pt])\prod\limits_{i=1}^m\tau_{d_i}(\gamma_i)\bigg)_\beta=(1+q)^2\cdot Z_P\bigg(\tilde X;q|\prod\limits_{i=1}^m\tau_{d_i}(p^*\gamma_i)\bigg)_{p^!\beta-e}.
\nan
\end{theorem}
\begin{theorem}\label{pt3}
Under the same assumptions as in Theorem \ref{pt1}, we have
\ban
Z_P\bigg(X;q|\tau_1([pt])\prod\limits_{i=1}^m\tau_{d_i}(\gamma_i)\bigg)_\beta=\frac{1}{2}(1-q^2)\cdot Z_P\bigg(\tilde X;q|\tau_0(-E^2)\prod\limits_{i=1}^m\tau_{d_i}(p^*\gamma_i)\bigg)_{p^!\beta-e}.
\nan
\end{theorem}

We also consider blow-up along a curve.

\begin{theorem}\label{curve0}
Let $p:\tilde X\rightarrow X$ be the blow-up along an irreducible, nonsingular embedded curve $C$ with $\int_Cc_1(X)\geqslant 0$. Suppose  that $\gamma_1,\cdots,\gamma_m\in H^{*}(X, {\mathbb Q})$ have supports away from $C$, and $d_1,\cdots,d_m\in\bZ_{\geqslant 0}$. Then for $\beta\in H_2(X, {\mathbb Z})$ and $k\in\bZ_{>0}$, we have
\ban
Z_P\bigg(\tilde X;q|\prod\limits_{i=1}^m\tau_{d_i}(p^*\gamma_i)\bigg)_{p^!\beta+ke}=0.
\nan
\end{theorem}
\begin{theorem}\label{curve1}
Let $p:\tilde X\rightarrow X$ be the blow-up along an irreducible, nonsingular embedded curve $C$ with $\int_Cc_1(X)>0$. Suppose  that $\gamma_1,\cdots,\gamma_m\in H^{*}(X, {\mathbb Q})$ have supports away from $C$, and $d_1,\cdots,d_m\in\bZ_{\geqslant 0}$. Then for nonzero $\beta\in H_2(X, {\mathbb Z})$, we have
\ban
 Z_P\bigg(X;q|\prod\limits_{i=1}^m\tau_{d_i}(\gamma_i)\bigg)_\beta =  Z_P\bigg(\tilde X;q|\prod\limits_{i=1}^m\tau_{d_i}(p^*\gamma_i)\bigg)_{p^!\beta}.
\nan
\end{theorem}
\begin{theorem}\label{curve2}
Let $p:\tilde X\rightarrow X$ be the blow-up along an irreducible, nonsingular embedded curve $C$ with $\int_Cc_1(X)>1$. Suppose  that $\gamma_1,\cdots,\gamma_m\in H^{*}(X, {\mathbb Q})$ have supports away from $C$, and $d_1,\cdots,d_m\in\bZ_{\geqslant 0}$. Then for nonzero $\beta\in H_2(X, {\mathbb Z})$, we have
\ban
 Z_P\bigg(X;q|\tau_0([C])\prod\limits_{i=1}^m\tau_{d_i}(\gamma_i)\bigg)_\beta=(1+q)\cdot Z_P\bigg(\tilde X;q|\prod\limits_{i=1}^m\tau_{d_i}(p^*\gamma_i)\bigg)_{p^!\beta-e}.
\nan
\end{theorem}
\begin{remark}
If deg$\gamma_i>2$, then $\gamma_i$ has support away from $C$.
\end{remark}

%In \cite{PT}, $\bZ$-valued BPS state counts for stable pair invariants with primary field insertions were introduced. Our blow-up formulae imply the following simple relations between BPS state counts of $X$ and those of $\tilde X$.

%\begin{corollary}
%\begin{itemize}
%\item[(a)]Under the same assumptions as in Theorem \ref{pt0}, we have
%\ban
%n_{g,p^!\beta+ke}^{\tilde X}(\gamma_1,\cdots,\gamma_m)=0.
%\nan
%\item[(b)]Under the same assumptions as in Theorem \ref{pt1}, we have 
%\ban
%n_{g,\beta}^X(\gamma_1,\cdots,\gamma_m)&=&n_{g,p^!\beta}^{\tilde X}(p^*\gamma_1,\cdots,p^*\gamma_m),\\
%n_{g,\beta}^X([pt],\gamma_1,\cdots,\gamma_m)&=&n_{g,p^!\beta-e}^{\tilde X}(p^*\gamma_1,\cdots,p^*\gamma_m).
%\nan
%\item[(c)]Under the same assumptions as in Theorem \ref{curve0}, we  have
%\ban
%n_{g,p^!\beta+ke}^{\tilde X}(\gamma_1,\cdots,\gamma_m)=0.
%\nan
%\item[(d)]Under the same assumptions as in Theorem \ref{curve1}, we have
%\ban
%n_{g,\beta}^X(\gamma_1,\cdots,\gamma_m)&=&n_{g,p^!\beta}^{\tilde X}(p^*\gamma_1,\cdots,p^*\gamma_m).
%\nan
%\item[(e)]Under the same assumptions as in Theorem \ref{curve2}, we have
%\ban
%n_{g,\beta}^X([C],\gamma_1,\cdots,\gamma_m)&=&n_{g,p^!\beta-e}^{\tilde X}(p^*\gamma_1,\cdots,p^*\gamma_m).
%\nan
%\end{itemize}
%\end{corollary}

The key tool used in this paper is the degeneration formula \cite{IP, Li, LW, LR, MPT}. Degeneration formula is powerful in the study of structures of Gromov-Witten, Donaldson-Thomas and stable pair theories \cite{HLR, LHH,  MOOP, OP, PP1, PP2, PP3, PP4, PP5}. In this paper, the blow-ups of $X$ can be described in terms of semi-stable degenerations of $X$, and we use degeneration formula to express invariants of $X$ and $\tilde X$ in terms of relative invariants of $(\tilde X, E)$. Then we use virtual localization \cite{GP} and degenerate contribution computation \cite{PT} to compute the relevant coefficients to obtain our results.

In \cite{HHKQ}, W. He, J. Hu, X. Qi and the author have obtained several blow-up formulae for all genera Gromov-Witten invariants for symplectic manifolds in real dimension six. Assuming GW/P corrspondence, many of the results of this paper can be derived from those of \cite{HHKQ}. Moreover, the corresponding results also hold in Donaldson-Thomas theory (except Theorem \ref{curve2}). The reason behind the similarity of blow-up formulae for Gromov-Witten, Donaldson-Thomas and stable pair invariants is that the behavior of these invariants under degeneraton is similar. Our blow-up formulae give positive evidence for GW/DT/P correspondence. Also, based on known results, our blow-up formulae give partial GW/DT/P correspondence for projective $3$-folds not necessarily toric or complete intersections in products of projective spaces.

The author is not able to prove the corresponding result of Theorem \ref{curve2} in Donaldson-Thomas theory. This is because in Donaldson-Thomas theory, free points are allowed to move in the whole variety, which makes the degeneration contribution computation difficult.

In \cite{HL}, J. Hu and W.-P. Li have studied the change of Donaldson-Thomas invarians under ordinary flops and extremal small transitions via degeneration formua. We can also study the change of stable pair invariants under these surgeries, using exactly the same arguments as in \cite{HL} to obtain similar results.

The rest of the paper is arranged as follows. In Section 2, we briefly review basic materials of absolute/relative stable pair invariants and the degeneration formula. In Section 3, we consider the case of blow-up at a point. In Section 4, we consier the case of blow-up along a curve.

\section{Preliminaries}

In this section, we briefly review absolute/relative stable pair invariants and the
degeneration formula and fix notations throughout.  We refer readers to \cite{LW, MPT, PT} for details.

A stable pair $(F,s)$ on $X$ consists of a pure sheaf $F$ on $X$ supported on a (possibly disconnected) Cohen-Macaulay curve and a section $s\in H^0(X,F)$ with zero dimensional cokernel. For $n\in\bZ$ and nonzero $\beta\in H_2(X,\bZ)$, let $P_n(X,\beta)$ be the moduli space of stable pairs $(F,s)$ with $\chi(F)=n$ and $[F]=\beta$. From the deformation theory of complexes in the derived category, the moduli space $P_n(X,\beta)$ carries a virtual fundamental class.

For $d\in\bZ_{\geqslant 0}$ and $\gamma\in H^*(X,\bZ)$, the descendant insertion $\tau_d(\gamma)$ is defined as follows. Let 
\ban
\pi_X:X\times P_n(X,\beta)&\rightarrow&X,\\
\pi_P:X\times P_n(X,\beta)&\rightarrow&P_n(X,\beta)
\nan
be tautological projections. Let $\bF$ be the universal sheaf over $X\times P_n(X,\beta)$. The operation
\ban
\pi_{P*}\bigg(\pi_X^*(\gamma)\cdot\textrm{ch}_{2+d}(\bF)\cap\pi_P^*(\cdot)\bigg):H_*(P_n(X,\beta),\bZ)\rightarrow H_*(P_n(X,\beta),\bZ)
\nan
is the action of $\tau_d(\gamma)$. The stable pair invariants with descendant insertions are defined as the virtual integration
\ban
\<\prod\limits_{i=1}^m\tau_{d_i}(\gamma_i)\>_{n,\beta}&=&\int_{P_n(X,\beta)}\prod\limits_{i=1}^m\tau_{d_i}(\gamma_i)\bigg([P_n(X,\beta)]^{vir}\bigg),
\nan
where $d_1\cdots,d_m\in\bZ_{\geqslant 0}$, and $\gamma_1,\cdots,\gamma_m\in H^*(X,\bZ)$. Denote the partition function of stable pair invariants as
\ban
Z_P\bigg(X;q|\prod\limits_{i=1}^m\tau_{d_i}(\gamma_i)\bigg)_\beta=\sum\limits_{n\in\bZ}\<\prod\limits_{i=1}^m\tau_{d_i}(\gamma_i)\>_{n,\beta}q^n.
\nan

Let $S\subset X$ be a nonsingular divisor. For $n\in\bZ$ and nonzero $\beta\in H_2(X,\bZ)$ with $\int_\beta[S]\geqslant 0$, let $P_n(X/S,\beta)$ be the moduli space of relative stable pairs, which carries a virtual fundamental class of degree $\int_\beta c_1(X)$. We have the following natural morphism
\ban
\epsilon:P_n(X/S,\beta)\rightarrow\textrm{Hilb}(S,\int_\beta[S])
\nan
The pull-back of cohomology classes of Hilb$(S,\int_\beta[S])$ gives relative insertions. 

Let us briefly recall Nakajima basis for the cohomology of Hilbert schemes of points of $S$. Let $\{\delta_i\}$ be a basis of $H^*(S,\bQ)$ with dual basis $\{\delta^i\}$. For any cohomology weighted partition $\eta$ with respect to the basis $\{\delta_i\}$, Nakajima constructed a cohomology class $C_\eta\in H^*(\textrm{Hilb}(S,|\eta|),\bQ)$. The Nakajima basis of  $H^*(\textrm{Hilb}(S,d),\bQ)$ is the set $\{C_\eta\}_{|\eta|=d}$. We refer readers to \cite{Na} for more details.

The partition function of relative stable pair invariants are defined by
\ban
Z_P\bigg(X/S;q|\prod\limits_{i=1}^m\tau_{d_i}(\gamma_i)|\eta\bigg)_\beta=\sum\limits_{n\in\bZ}q^n\int_{[P_n(X/S,\beta)]^{vir}}\prod\limits_{i=1}^m\tau_{d_i}(\gamma_i)\cdot\epsilon^*C_\eta.
\nan

Let $\pi : \chi \rightarrow\bA^1$ be a nonsingular $4$-fold over $\bA^1$ such that $\chi_t = \pi^{-1}(t) \cong X $ for $t\not= 0$ and $\chi_0$ is a union of two irreducible nonsingular projective $3$-folds $X_1$
and $X_2$ intersecting transversally along a nonsigular projective surface $S$.

Consider the natural inclusion maps
$$
  i_t: X=\chi_t \longrightarrow \chi,\,\,\,\,\,\,\,\,
  i_0:\chi_0\longrightarrow \chi,
$$
and the gluing map
$$
  g= (j_1,j_2) : X_1\coprod X_2\longrightarrow \chi_0.
$$
We have
$$
  H_2(X,\bZ)\stackrel{i_{t*}}{\longrightarrow}
  H_2(\chi,\bZ)\stackrel{i_{0_*}}{\longleftarrow}
  H_2(\chi_0,\bZ)\stackrel{g_*}{\longleftarrow} H_2(X_1,\bZ)\oplus
  H_2(X_2,\bZ),
$$
where $i_{0*}$ is an isomorphism since there exists a deformation retract from $\chi$ to $\chi_0$(see \cite{Cl}). Also, since the family $\chi\rightarrow\bA^1$ comes from a trivial family, it follows that each $\gamma\in H^*(X,\bQ) $ has
global liftings such that the restriction $\gamma(t)$ on $\chi_t$ is defined for all $t$.

The degeneration formula for stable pair theory expresses absolute invariants of $X$ via relative invariants of $(X_1,S)$ and $(X_2,S)$:
\ban
&&Z_P\bigg(X;q|\prod\limits_{i=1}^m\tau_{d_i}(\gamma_i)\bigg)_\beta\\
&=&\sum Z_P\bigg(X_1/S;q|\prod\limits_{i\in P_1}\tau_{d_i}(j_1^*\gamma_i(0))|\eta\bigg)_{\beta_1}\cdot\frac{(-1)^{|\eta|-\ell(\eta)}\fz(\eta)}{q^{|\eta|}}\cdot Z_P\bigg(X_2/S;q|\prod\limits_{i\in P_2}\tau_{d_i}(j_2^*\gamma_i(0))|\eta^\vee\bigg)_{\beta_2},
\nan
where $\fz(\eta)=|\textrm{Aut}(\eta)|\cdot\prod\limits_{i=1}^{\ell(\eta)}\eta_i$, $\eta^\vee$ is defined by taking the Poincar\'e duals of the cohomology weights of $\eta$, and the sum is over cohomology weighted partitions $\eta$, degree splittings $i_{t*}\beta=i_{0*}(j_{1*}\beta_1+j_{2*}\beta_2)$, and marking partitions $P_1\coprod P_2=\{1,\cdots,m\}$. In particular, if $(\eta,\beta_1,\beta_2)$ has nontrivial contribution in the degeneration formula, then we have the following dimension constraint:
\ban
v\textrm{dim}_\bC P_n(X_1/S,\beta_1)+v\textrm{dim}_\bC P_n(X_2/S,\beta_2)=v\textrm{dim}_\bC P_n(X,\beta)+2|\eta|.
\nan

\section{Formulae for blow-up at a point}

In this section, we consider blow-up at a point and prove Theorem \ref{pt0}, \ref{pt1}, \ref{pt2}, and \ref{pt3}. We always assume that total degrees of insertions match the virtual dimensions of the moduli spaces, since otherwise the required equalities are trivial.

Throughout this section, we let $H$ be the hyperplane at infinity in $\bP^3$, and $\tilde\bP^3$ is the blow-up of $\bP^3$ at a point not in $H$.

We first prove Theorem \ref{pt0}. Degenerate $\tilde X$ along $E$, and by the degeneration formula, we have
\begin{eqnarray}\label{dpt0}
&&Z_P\bigg(\tilde X;q|\prod\limits_{i=1}^m\tau_{d_i}(p^*\gamma_i)\bigg)_{p^!\beta+ke}\\
&=&\sum Z_P(\tilde\bP^3/H;q||\eta)_{\beta_1}\cdot\frac{(-1)^{|\eta|-\ell(\eta)}\fz(\eta)}{q^{|\eta|}}\cdot Z_P\bigg(\tilde X/E;q|\prod\limits_{i=1}^m\tau_{d_i}(p^*\gamma_i)|\eta^\vee\bigg)_{\beta_2}, \nonumber
\end{eqnarray}
where we have assumed that the class $p^*\gamma_i$ has support away from $E$. By our assumption that degrees match the virtual dimensions, we have
\ban
v\textrm{dim}_\bC P_n(\tilde X,p^!\beta)=\frac{1}{2}\sum\limits_{i=1}^{m}\gamma_i+\sum\limits_{i=1}^md_i-m.
\nan

Suppose that $(\eta=\{(\eta_i,\delta_{j_i})\}_{i=1}^{\ell(\eta)},\beta_1,\beta_2)$ has nonzero contribution in \eqref{dpt0}. Then
\ban
v\textrm{dim}_\bC P_n(\tilde\bP^3/H,\beta_1)&=&\int_{\beta_1}c_1(\tilde\bP^3),\\
v\textrm{dim}_\bC P_n(\tilde X/E,\beta_2)&=&\frac{1}{2}\sum_{i=1}^m deg\gamma_i+\sum_{i=1}^md_i+\frac{1}{2}\sum\limits_{i=1}^{\ell(\eta)}deg\delta^{j_i}-\ell(\eta)+|\eta|-m.
\nan
So by the dimension constraint,
\ban
\frac 12\sum\limits_{i=1}^{\ell(\eta)}deg\delta^{j_i}+\int_{\beta_1}c_1(\tilde{\bP}^3)-|\eta|=\ell(\eta).
\nan

Let $L\in H_2(\tilde\bP^3,\bZ)$ be the class of the total transform of a line in $\bP^3$. Then we have the following natural decomposition
\ban
H_2(\tilde\bP^3,\bZ)=\bZ F\oplus\bZ L.
\nan
We have the following constraint for $\beta_1$:
\ban
\left\{\begin{array}{lcl}\beta_1\cdot H&=&|\eta|,\\\beta_1\cdot E&=&-k.\end{array}\right.
\nan
So we have $\int_{\beta_1}c_1(\tilde\bP^3)=4|\eta|+2k$. Now the dimension constraint becomes
\ban
\frac 12\sum\limits_{i=1}^{\ell(\eta)}deg\delta^{j_i}+3|\eta|+2k=\ell(\eta).
\nan
We observe that no partition satisfies the dimension constraint, and this proves Theorem \ref{pt0}.

Next, we prove Theorem \ref{pt1}. We divide the proof of Theorem \ref{pt1} into two comparison lemmas of stable pair invariants.

\begin{lemma}\label{ptlemma1}
Under the assumptions as in Theorem \ref{pt1}, we have
\ban
Z_P\bigg(X;q|\prod\limits_{i=1}^m\tau_{d_i}(\gamma_i)\bigg)_\beta=Z_P\bigg(\tilde X/E;q|\prod\limits_{i=1}^m\tau_{d_i}(p^*\gamma_i)|\bigg)_{p^!\beta}.
\nan
\end{lemma}
\begin{proof}
Degenerate $X$ at a point $P$, and by the degeneration formula, we have
\ba\label{dptlemma1}
&&Z_P\bigg(X;q|\prod\limits_{i=1}^m\tau_{d_i}(\gamma_i)\bigg)_\beta\\
&=&\sum Z_P\bigg(\bP^3/H;q||\eta\bigg)_{\beta_1}\cdot\frac{(-1)^{|\eta|-\ell(\eta)}\fz(\eta)}{q^{|\eta|}}\cdot Z_P\bigg(\tilde X/E;q|\prod\limits_{i=1}^m\tau_{d_i}(p^*\gamma_i)|\eta^\vee\bigg)_{\beta_2},\nonumber
\na
where we have assumed that the support of $\gamma_i$ is away from $P$. By our assumption that total degrees of insertions match the virtual dimensions of moduli spaces, we have
\ban
v\textrm{dim}_\bC P_n(X,\beta)&=&\frac{1}{2}\sum_{i=1}^m deg\gamma_i+\sum_{i=1}^md_i-m.
\nan

Suppose that $(\eta=\{(\eta_i,\delta_{j_i})\}_{i=1}^{\ell(\eta)},\beta_1,\beta_2)$ has nonzero contribution in \eqref{dptlemma1}. Then 
\begin{eqnarray*}
v\textrm{dim}_\bC P_n(\bP^3/H,\beta_1)&=&\int_{\beta_1}c_1(\bP^3),\\
v\textrm{dim}_\bC P_n(\tilde X/E,\beta_2)&=&\frac{1}{2}\sum_{i=1}^m deg\gamma_i+\sum_{i=1}^md_i+\frac{1}{2}\sum\limits_{i=1}^{\ell(\eta)}deg\delta^{j_i}-\ell(\eta)+|\eta|-m.
\end{eqnarray*}
So by the dimension constraint,
\ban
\frac 12\sum\limits_{i=1}^{\ell(\eta)}deg\delta^{j_i}+\int_{\beta_1}c_1(\bP^3)-|\eta|=\ell(\eta).
\nan
Note that $\beta_1\cdot H=|\eta|$, and hence $\beta_1=|\eta|L$, which implies that
\ban
\int_{\beta_1}c_1(\bP^3)=4|\eta|.
\nan
Now the dimension constraint becomes
\ban
\frac 12\sum\limits_{i=1}^{\ell(\eta)}deg\delta^{j_i}+3|\eta|=\ell(\eta).
\nan
So the dimension constraint holds only if $\eta=\emptyset$, which implies Lemma \ref{ptlemma1}.
\end{proof}

\begin{lemma}\label{ptlemma2}
Under the assumptions as in Theorem \ref{pt1}, we have
\ban
Z_P\bigg(\tilde X;q|\prod\limits_{i=1}^m\tau_{d_i}(p^*\gamma_i)\bigg)_{p^!\beta}=Z_P\bigg(\tilde X/E;q|\prod\limits_{i=1}^m\tau_{d_i}(p^*\gamma_i)|\bigg)_{p^!\beta}.
\nan
\end{lemma}
\begin{proof}
Degenerate $\tilde X$ along $E$, and by the degeneration formula, we have
\begin{eqnarray}\label{dptlemma2}
&&Z_P\bigg(\tilde X;q|\prod\limits_{i=1}^m\tau_{d_i}(p^*\gamma_i)\bigg)_{p^!\beta}\\
&=&\sum Z_P(\tilde\bP^3/H;q||\eta)_{\beta_1}\cdot\frac{(-1)^{|\eta|-\ell(\eta)}\fz(\eta)}{q^{|\eta|}}\cdot Z_P\bigg(\tilde X/E;q|\prod\limits_{i=1}^m\tau_{d_i}(p^*\gamma_i)|\eta^\vee\bigg)_{\beta_2}, \nonumber
\end{eqnarray}
where we have assumed that the class $p^*\gamma_i$ has support away from $E$. By our assumption that degrees match the virtual dimensions, we have
\ban
v\textrm{dim}_\bC P_n(\tilde X,p^!\beta+ke)=\frac{1}{2}\sum\limits_{i=1}^{m}\gamma_i+\sum\limits_{i=1}^md_i-m.
\nan

Suppose that $(\eta=\{(\eta_i,\delta_{j_i})\}_{i=1}^{\ell(\eta)},\beta_1,\beta_2)$ has nonzero contribution in \eqref{dptlemma2}. Then
\ban
v\textrm{dim}_\bC P_n(\tilde\bP^3/H,\beta_1)&=&\int_{\beta_1}c_1(\tilde\bP^3),\\
v\textrm{dim}_\bC P_n(\tilde X/E,\beta_2)&=&\frac{1}{2}\sum_{i=1}^m deg\gamma_i+\sum_{i=1}^md_i+\frac{1}{2}\sum\limits_{i=1}^{\ell(\eta)}deg\delta^{j_i}-\ell(\eta)+|\eta|-m.
\nan
So by the dimension constraint,
\ban
\frac 12\sum\limits_{i=1}^{\ell(\eta)}deg\delta^{j_i}+\int_{\beta_1}c_1(\tilde{\bP}^3)-|\eta|=\ell(\eta).
\nan

Let $L\in H_2(\tilde\bP^3,\bZ)$ be the class of the total transform of a line in $\bP^3$. Then we have the following natural decomposition
\ban
H_2(\tilde\bP^3,\bZ)=\bZ F\oplus\bZ L.
\nan
We have the following constraint for $\beta_1$:
\ban
\left\{\begin{array}{lcl}\beta_1\cdot H&=&|\eta|,\\\beta_1\cdot E&=&0.\end{array}\right.
\nan
So we have $\beta_1=|\eta|L$, and hence $\int_{\beta_1}c_1(\tilde\bP^3)=4|\eta|$. Now the dimension constraint becomes
\ban
\frac 12\sum\limits_{i=1}^{\ell(\eta)}deg\delta^{j_i}+3|\eta|=\ell(\eta).
\nan
So the dimension constraint holds only if $\eta=\emptyset$, which implies Lemma \ref{ptlemma2}.
\end{proof}

The above comparison results give Theorem \ref{pt1}.

To prove Theorem \ref{pt2}, we need the following two comparison lemmas.

\begin{lemma}\label{ptlemma3}
Under the same assumptions as in Theorem \ref{pt2}, we have
\ban
&&Z_P\bigg(X;q|\tau_0([pt])\prod\limits_{i=1}^m\tau_{d_i}(\gamma_i)\bigg)_\beta\\
&=&Z_P\bigg(\bP^3/H;q|\tau_0([pt])|(1,[pt])\bigg)_L\cdot\frac{1}{q}\cdot Z_P\bigg(\tilde X/E;q|\prod\limits_{i=1}^m\tau_{d_i}(p^*\gamma_i)|(1,\mathbbm 1)\bigg)_{p^!\beta-e},
\nan
where $L$ is the class of a line.
\end{lemma}
\begin{proof}
Degenerate $X$ at a point $P$, and by the degeneration formula, we have
\ba\label{dptlemma3}
&&Z_P\bigg(X;q|\tau_0([pt])\prod\limits_{i=1}^m\tau_{d_i}(\gamma_i)\bigg)_\beta\\
&=&\sum Z_P\bigg(\bP^3/H;q|\tau_0([pt])|\eta\bigg)_{\beta_1}\cdot\frac{(-1)^{|\eta|-\ell(\eta)}\fz(\eta)}{q^{|\eta|}}\cdot Z_P\bigg(\tilde X/E;q|\prod\limits_{i=1}^m\tau_{d_i}(p^*\gamma_i)|\eta^\vee\bigg)_{\beta_2},\nonumber
\na
where we have assumed that the support of $\gamma_i$ is away from $P$. By our assumption that total degrees of insertions match the virtual dimensions of moduli spaces, we have
\ban
v\textrm{dim}_\bC P_n(X,\beta)&=&\frac{1}{2}\sum_{i=1}^m deg\gamma_i+\sum_{i=1}^md_i+2-m.
\nan

Suppose that $(\eta=\{(\eta_i,\delta_{j_i})\}_{i=1}^{\ell(\eta)},\beta_1,\beta_2)$ has nonzero contribution in \eqref{dptlemma3}. Then 
\begin{eqnarray*}
v\textrm{dim}_\bC P_n(\bP^3/H,\beta_1)&=&\int_{\beta_1}c_1(\bP^3),\\
v\textrm{dim}_\bC P_n(\tilde X/E,\beta_2)&=&\frac{1}{2}\sum_{i=1}^m deg\gamma_i+\sum_{i=1}^md_i+\frac{1}{2}\sum\limits_{i=1}^{\ell(\eta)}deg\delta^{j_i}-\ell(\eta)+|\eta|-m.
\end{eqnarray*}
So by the dimension constraint,
\ban
\frac 12\sum\limits_{i=1}^{\ell(\eta)}deg\delta^{j_i}+\int_{\beta_1}c_1(\bP^3)-|\eta|=2+\ell(\eta).
\nan
Note that $\beta_1\cdot H=|\eta|$, and hence $\beta_1=|\eta|L$, which implies that
\ban
\int_{\beta_1}c_1(\bP^3)=4|\eta|.
\nan
Now the dimension constraint becomes
\ban
\frac 12\sum\limits_{i=1}^{\ell(\eta)}deg\delta^{j_i}+3|\eta|=2+\ell(\eta).
\nan
So the dimension constraint holds only if $\eta=(1,[pt])$, which implies Lemma \ref{ptlemma3}.
\end{proof}

\begin{lemma}\label{ptlemma4}
\ban
&&Z_P\bigg(\tilde X;q|\prod\limits_{i=1}^m\tau_{d_i}(p^*\gamma_i)\bigg)_{p^!\beta-e}\\
&=&Z_P\bigg(\tilde P^3/H;q||(1,[pt])\bigg)_{F}\cdot\frac{1}{q}\cdot Z_P\bigg(\tilde X/E;q|\prod\limits_{i=1}^m\tau_{d_i}(p^*\gamma_i)|(1,\mathbbm 1)\bigg)_{p^!\beta-e},
\nan
where $F$ is the fiber class of $\tilde P^3\cong\bP_H(\cO(1)\oplus\cO)$.
\end{lemma}
\begin{proof}
Degenerate $\tilde X$ along $E$, and by the degeneration formula, we have
\begin{eqnarray}\label{dptlemma4}
&&Z_P\bigg(\tilde X;q|\prod\limits_{i=1}^m\tau_{d_i}(p^*\gamma_i)\bigg)_{p^!\beta-e}\\
&=&\sum Z_P(\tilde\bP^3/H;q||\eta)_{\beta_1}\cdot\frac{(-1)^{|\eta|-\ell(\eta)}\fz(\eta)}{q^{|\eta|}}\cdot Z_P\bigg(\tilde X/E;q|\prod\limits_{i=1}^m\tau_{d_i}(p^*\gamma_i)|\eta^\vee\bigg)_{\beta_2}, \nonumber
\end{eqnarray}
where we have assumed that the class $p^*\gamma_i$ has support away from $E$. By our assumption that degrees match the virtual dimensions, we have
\ban
v\textrm{dim}_\bC P_n(\tilde X,p^!\beta-e)=\frac{1}{2}\sum\limits_{i=1}^{m}\gamma_i+\sum\limits_{i=1}^md_i-m.
\nan

Suppose that $(\eta=\{(\eta_i,\delta_{j_i})\}_{i=1}^{\ell(\eta)},\beta_1,\beta_2)$ has nonzero contribution in \eqref{dptlemma4}. Then
\ban
v\textrm{dim}_\bC P_n(\tilde\bP^3/H,\beta_1)&=&\int_{\beta_1}c_1(\bP^3),\\
v\textrm{dim}_\bC P_n(\tilde X/E,\beta_2)&=&\frac{1}{2}\sum_{i=1}^m deg\gamma_i+\sum_{i=1}^md_i+\frac{1}{2}\sum\limits_{i=1}^{\ell(\eta)}deg\delta^{j_i}-\ell(\eta)+|\eta|-m.
\nan
So by the dimension constraint,
\ban
\frac 12\sum\limits_{i=1}^{\ell(\eta)}deg\delta^{j_i}+\int_{\beta_1}c_1(\tilde{\bP}^3)-|\eta|=\ell(\eta).
\nan

Let $L\in H_2(\tilde\bP^3,\bZ)$ be the class of the total transform of a line in $\bP^3$. Then we have the following natural decomposition
\ban
H_2(\tilde\bP^3,\bZ)=\bZ F\oplus\bZ L.
\nan
We have the following constraint for $\beta_1$:
\ban
\left\{\begin{array}{lcl}\beta_1\cdot H&=&|\eta|,\\\beta_1\cdot E&=&1.\end{array}\right.
\nan
So we have $\beta_1=F+(|\eta|-1)L$, and hence $\int_{\beta_1}c_1(\tilde\bP^3)=4|\eta|-2$. Now the dimension constraint becomes
\ban
\frac 12\sum\limits_{i=1}^{\ell(\eta)}deg\delta^{j_i}+3|\eta|=2+\ell(\eta).
\nan
So the dimension constraint holds only if $\eta=(1,[pt])$, which implies Lemma \ref{ptlemma4}.
\end{proof}

{\bf Proof of Theorem \ref{pt2}:} By Lemma \ref{ptlemma3} and \ref{ptlemma4}, in the particular case $X=\bP^3$, we have
\ban
\frac{Z_P\bigg(\bP^3;q|\tau_0([pt])^2\bigg)_L}{Z_P\bigg(\tilde\bP^3;q|\tau_0([pt])\bigg)_{F}}=\frac{Z_P\bigg(\bP^3/H;q|\tau_0([pt])|(1,[pt])\bigg)_L}{Z_P\bigg(\tilde\bP^3/H;q||(1,[pt])\bigg)_{F}}.
\nan
Now by virtual localization \cite{GP} or by (4.2) in \cite{PT}, we have
\ban
Z_P\bigg(\bP^3;q|\tau_0([pt])^2\bigg)_L&=&q(1+q)^2,\\
Z_P\bigg(\tilde\bP^3;q|\tau_0([pt])\bigg)_{F}&=&q,
\nan
which gives Theorem \ref{pt2}.

Theorem \ref{pt3} relies on the following Lemma \ref{ptlemma5} and \ref{ptlemma6}, the proof of which is analogous to that of Lemma \ref{ptlemma3}, \ref{ptlemma4} respectively. 

\begin{lemma}\label{ptlemma5}
\ban
&&Z_P\bigg(X;q|\tau_1([pt])\prod\limits_{i=1}^m\tau_{d_i}(\gamma_i)\bigg)_\beta\\
&=&Z_P\bigg(\bP^3/H;q|\tau_1([pt])|(1,[L])\bigg)_L\cdot\frac{1}{q}\cdot Z_P\bigg(\tilde X/E;q|\prod\limits_{i=1}^m\tau_{d_i}(p^*\gamma_i)|(1,[L])\bigg)_{p^!\beta-e}.
\nan
\end{lemma}

\begin{lemma}\label{ptlemma6}
\ban
&&Z_P\bigg(\tilde X;q|\tau_0(-E^2)\prod\limits_{i=1}^m\tau_{d_i}(p^*\gamma_i)\bigg)_{p^!\beta-e}\\
&=&Z_P\bigg(\tilde X/H;q|\tau_0(-E^2)|(1,[L])\bigg)_F\cdot\frac{1}{q}\cdot Z_P\bigg(\tilde X/E;q|\prod\limits_{i=1}^m\tau_{d_i}(p^*\gamma_i)|(1,[L])\bigg)_{p^!\beta-e}.
\nan
\end{lemma}

{\bf Proof of Theorem \ref{pt3}:} By Lemma \ref{ptlemma5} and \ref{ptlemma6}, in the particular case $X=\bP^3$, we have
\ban
\frac{Z_P\bigg(\bP^3;q|\tau_1([pt])\tau_0([L])\bigg)_L}{Z_P\bigg(\bP^3;q|\tau_0(-E^2)\tau_0([L])\bigg)_{F}}=\frac{Z_P\bigg(\bP^3/H;q|\tau_1([pt])|(1,[L])\bigg)_L}{Z_P\bigg(\tilde\bP^3/H;q|\tau_0(-E^2)|(1,[L])\bigg)_F},
\nan
By virtual localization \cite{GP}, we have
\ban
Z_P\bigg(\bP^3;q|\tau_1([pt])\tau_0([L])\bigg)_L&=&\frac{1}{2}q(1-q^2),\\
Z_P\bigg(\bP^3;q|\tau_0(-E^2)\tau_0([L])\bigg)_{F}&=&q,
\nan
which gives Theorem \ref{pt3}.

\section{Formulae for blow-up along a curve}

In this section, we consider blow-up along a nonsingular embedded curve and prove Theorem \ref{curve0}, \ref{curve1} and \ref{curve2}.
We always assume that total degrees of insertions match the virtual dimensions of the moduli spaces, since otherwise the required equalities are trivial.

Throughout this section, we let $N_C$  be the normal bundle of $C$ in $X$, and $N_E$ the normal bundle of the exceptional divisor $E$ in $\tilde X$.

We first prove Theorem \ref{curve0}. Degenerate $\tilde X$ along $E$, and by the degeneration formula, we have
\begin{eqnarray}\label{dcurvelemma0}
&&Z_P\bigg(\tilde X;q|\prod\limits_{i=1}^m\tau_{d_i}(p^*\gamma_i)\bigg)_{p^!\beta+ke}\\
&=&\sum Z_P\bigg(\bP_E(N_E\oplus\cO_E)/D_\infty;q||\eta\bigg)_{\beta_1}\cdot\frac{(-1)^{|\eta|-\ell(\eta)}\fz(\eta)}{q^{|\eta|}}\nonumber\\
&&\qquad\cdot Z_P\bigg(\tilde X/E;q|\prod\limits_{i=1}^m\tau_{d_i}(p^*\gamma_i)|\eta^\vee\bigg)_{\beta_2},\nonumber
\end{eqnarray}
where we have assumed that the support of $p^*\gamma_i$ is away from $E$, and $D_\infty=\bP_E(N_E\oplus\{0\})$. Recall that we have assumed that
\ban
v\textrm{dim}_\bC P_n(\tilde X,p^!\beta+ke)=\frac{1}{2}\sum\limits_{i=1}^mdeg\gamma_i+\sum\limits_{i=1}^md_i-m.
\nan

Assume that $(\eta=\{(\eta_i,\delta_{j_i})\}_{i=1}^{\ell(\eta)},\beta_1,\beta_2)$ has nontrivial contribution in \eqref{dcurvelemma0}, and then
\begin{eqnarray*}
&&v\textrm{dim}_\bC P_n(\bP_E(N_E\oplus\cO_E)/D_\infty,\beta_1)=\int_{\beta_1}c_1(\bP_E(N_E\oplus\cO_E)),\\
&&v\textrm{dim}_\bC P_n(\tilde X/E,\beta_2)=\frac{1}{2}\sum_{i=1}^m deg\gamma_i+\sum_{i=1}^md_i+\frac{1}{2}\sum\limits_{i=1}^{\ell(\eta)}deg\delta^{j_i}-\ell(\eta)+|\eta|-m.
\end{eqnarray*}
So by the dimension constraint,
\begin{eqnarray*}
\frac{1}{2}\sum\limits_{i=1}^{\ell(\eta)}\textrm{deg}\delta^{j_i}+\int_{\beta_1}c_1(\bP_E(N_E\oplus\cO_E))-|\eta|=\ell(\eta).\nonumber
\end{eqnarray*}
Let $\xi$ be the tautological line bundle of $\bP_E(N_E\oplus\cO_E)$. Then Euler exact sequence gives
\begin{eqnarray*}
c_1(\bP_E(N_E\oplus\cO_E))=\pi^*c_1(E)+\pi^*c_1(N_E)-2c_1(\xi),
\end{eqnarray*}
where $\pi:\bP_E(N_E\oplus\cO_E)\rightarrow E$ is the natural projection. Note that $N_E$ is the tautological line bundle of $E\cong\bP_C(N_C)$, and so
\begin{eqnarray*}
c_1(E)=\pi_E^*c_1(X)|_C-2c_1(N_E),
\end{eqnarray*}
where $\pi_E:E\rightarrow C$ is the natural projection. Therefore,
\begin{eqnarray*}
c_1(\bP_E(N_E\oplus\cO_E))=(\pi_E\circ\pi)^*c_1(X)|_C-\pi^*c_1(N_E)-2c_1(\xi).
\end{eqnarray*}
Note that we have the following natural decomposition
\begin{eqnarray*}
H_2(\bP_E(N_E\oplus\cO_E),\bZ)\cong\bZ F\oplus H_2(E,\bZ),
\end{eqnarray*} 
and we can write
\begin{eqnarray*}
\beta_1=aF+\pi_*\beta_1,\textrm{ for some }a\in\bZ_{\geqslant 0}.
\end{eqnarray*}
We have the following constraints for $\beta_1$:
\begin{eqnarray*}
\left\{\begin{array}{ccl}\beta_1\cdot D_\infty&=&|\eta|,\\\beta_1\cdot E&=&-k,\end{array}\right.
\end{eqnarray*}
and this gives 
\begin{eqnarray*}
\pi_*\beta_1\cdot E=-|\eta|-k.
\end{eqnarray*}
Note that $-c_1(\xi)$ is the Poincar\'e dual of the divisor $D_\infty$ in $\bP_E(N_E\oplus\cO_E)$, and therefore
\begin{eqnarray*}
\int_{\beta_1}c_1(\bP_E(N_E\oplus\cO_E))=\int_{(\pi_E\circ\pi)_*\beta_1}c_1(X)|_C+3|\eta|+k.
\end{eqnarray*}
Hence the dimension constraint becomes
\begin{eqnarray*}
\frac{1}{2}\sum\limits_{i=1}^{\ell(\eta)}deg\delta^{j_i}+\int_{(\pi_E\circ\pi)_*\beta_1}c_1(X)|_C+2|\eta|+k=\ell(\eta).\nonumber
\end{eqnarray*}
We observe that no partition satisfies the dimension constraint, which gives Theorem \ref{curve0}.

Next, we prove Theorem \ref{curve1}. We divide the proof into two comparison lemmas of stable pair invariants.

\begin{lemma}\label{curvelemma1}
Under the same assumptions as in THeorem \ref{curve1}, we have
\ban
 Z_P\bigg(X;q|\prod\limits_{i=1}^m\tau_{d_i}(\gamma_i)\bigg)_\beta=Z_P\bigg(\tilde X/E;q|\prod\limits_{i=1}^m\tau_{d_i}(p^*\gamma_i)|\bigg)_{p^!\beta}.
\nan
\end{lemma}
\begin{proof}
Degenerate $X$ along $C$, and by the degeneration formula, we have:
\begin{eqnarray}\label{dcurvelemma1}
&&Z_P\bigg(X;q|\prod\limits_{i=1}^m\tau_{d_i}(\gamma_i)\bigg)_\beta\\
&=&\sum Z_P\bigg(\bP_C(N_C\oplus\cO_C)/D_\infty;q||\eta\bigg)_{\beta_1}\cdot\frac{(-1)^{|\eta|-\ell(\eta)}\fz(\eta)}{q^{|\eta|}}\nonumber\\
&&\qquad\cdot Z_P\bigg(\tilde X/E;q|\prod\limits_{i=1}^m\tau_{d_i}(p^*\gamma_i)|\eta^\vee\bigg)_{\beta_2}, \nonumber
\end{eqnarray}
where we have assumed that the support of $\gamma_i$ is away from $C$, and $D_\infty=\bP_C(N_C\oplus\{0\})$. Recall that we have assumed that 
\ban
v\textrm{dim}_\bC P_n(X,\beta)=\frac{1}{2}\sum\limits_{i=1}^mdeg\gamma_i+\sum\limits_{i=1}^md_i-m.
\nan

Suppose that $(\eta=\{(\eta_i,\delta_{j_i})\}_{i=1}^{\ell(\eta)},\beta_1,\beta_2)$ has nonzero contribution in \eqref{dcurvelemma1}. Then
\ban
&&v\textrm{dim}_\bC P_n(\bP_C(N_C\oplus\cO_C)/D_\infty,\beta_1)=\int_{\beta_1}c_1(\bP_C(N_C\oplus\cO_C)),\\
&&v\textrm{dim}_\bC P_n(\tilde X/E,\beta_2)=\frac{1}{2}\sum_{i=1}^m deg\gamma_i+\sum_{i=1}^md_i+\frac{1}{2}\sum\limits_{i=1}^{\ell(\eta)}deg\delta^{j_i}-\ell(\eta)+|\eta|-m.
\nan
So by the dimension constraint,
\begin{eqnarray*}
\frac{1}{2}\sum\limits_{i=1}^{\ell(\eta)}deg\delta^{j_i}+\int_{\beta_1}c_1(\bP_C(N_C\oplus\cO_C))-|\eta|=\ell(\eta).\nonumber
\end{eqnarray*}
Let $\xi$ be the tautological line bundle of $\bP_C(N_C\oplus\cO_C)$, and we have
\begin{eqnarray*}
c_1(\bP_C(N_C\oplus\cO_C))=\pi^*c_1(X)|_C-3c_1(\xi),
\end{eqnarray*}
where $\pi:\bP_C(N_C\oplus\cO_C)\rightarrow C$ is the natural projection. Note that $-c_1(\xi)$ is the Poincar\'e dual of the divisor $D_\infty$ in $\bP_C(N_C\oplus\cO_C)$. Since $|\eta|=\beta_1\cdot D_\infty$, it follows that
\begin{eqnarray*}
\int_{\beta_1}c_1(\bP_C(N_C\oplus\cO_C))=\int_{\pi_*\beta_1}c_1(X)|_C+3|\eta|.
\end{eqnarray*}
Therefore, dimension constraint becomes
\begin{eqnarray*}
&&\frac{1}{2}\sum\limits_{i=1}^{\ell(\eta)}deg\delta^{j_i}+\int_{\pi_*\beta_1}c_1(X)|_C+2|\eta|=\ell(\eta).
\end{eqnarray*}
The dimension constraint holds only if 
\begin{eqnarray*}
\eta=\emptyset,\quad\int_{\pi_*\beta_1}c_1(X)|_C=0,
\end{eqnarray*}
which implies Lemma \ref{curvelemma1}.
\end{proof}

\begin{lemma}\label{curvelemma2}
Under the same assumptions as in Theorem \ref{curve1}, we have
\ban
Z_P\bigg(\tilde X;q|\prod\limits_{i=1}^m\tau_{d_i}(p^*\gamma_i)\bigg)_{p^!\beta}=Z_P\bigg(\tilde X/E;q|\prod\limits_{i=1}^m\tau_{d_i}(p^*\gamma_i)|\bigg)_{p^!\beta}.
\nan
\end{lemma}
\begin{proof}
Degenerate $\tilde X$ along $E$, and by the degeneration formula, we have
\begin{eqnarray}\label{dcurvelemma2}
&&Z_P\bigg(\tilde X;q|\prod\limits_{i=1}^m\tau_{d_i}(p^*\gamma_i)\bigg)_{p^!\beta}\\
&=&\sum Z_P\bigg(\bP_E(N_E\oplus\cO_E)/D_\infty;q||\eta\bigg)_{\beta_1}\cdot\frac{(-1)^{|\eta|-\ell(\eta)}\fz(\eta)}{q^{|\eta|}}\nonumber\\
&&\qquad\cdot Z_P\bigg(\tilde X/E;q|\prod\limits_{i=1}^m\tau_{d_i}(p^*\gamma_i)|\eta^\vee\bigg)_{\beta_2},\nonumber
\end{eqnarray}
where we have assumed that the support of $p^*\alpha_i$ is away from $E$, and $D_\infty=\bP_E(N_E\oplus\{0\})$. Recall that we have assumed that
\ban
v\textrm{dim}_\bC P_n(\tilde X,p^!\beta)=\frac{1}{2}\sum\limits_{i=1}^mdeg\gamma_i+\sum\limits_{i=1}^md_i-m.
\nan

Assume that $(\eta=\{(\eta_i,\delta_{j_i})\}_{i=1}^{\ell(\eta)},\beta_1,\beta_2)$ has nontrivial contribution in \eqref{dcurvelemma2}, and then
\begin{eqnarray*}
&&v\textrm{dim}_\bC P_n(\bP_E(N_E\oplus\cO_E)/D_\infty,\beta_1)=\int_{\beta_1}c_1(\bP_E(N_E\oplus\cO_E)),\\
&&v\textrm{dim}_\bC P_n(\tilde X/E,\beta_2)=\frac{1}{2}\sum_{i=1}^m deg\gamma_i+\sum_{i=1}^md_i+\frac{1}{2}\sum\limits_{i=1}^{\ell(\eta)}deg\delta^{j_i}-\ell(\eta)+|\eta|-m.
\end{eqnarray*}
So by the dimension constraint,
\begin{eqnarray*}
\frac{1}{2}\sum\limits_{i=1}^{\ell(\eta)}\textrm{deg}\delta^{j_i}+\int_{\beta_1}c_1(\bP_E(N_E\oplus\cO_E))-|\eta|=\ell(\eta).\nonumber
\end{eqnarray*}
Following the notations in the proof of Theorem \ref{curve0}, we have the following constraints for $\beta_1$:
\begin{eqnarray*}
\left\{\begin{array}{ccl}\beta_1\cdot D_\infty&=&|\eta|,\\\beta_1\cdot E&=&0,\end{array}\right.
\end{eqnarray*}
and this gives 
\begin{eqnarray*}
\pi_*\beta_1\cdot E=-|\eta|,
\end{eqnarray*}
which implies that
\begin{eqnarray*}
\int_{\beta_1}c_1(\bP_E(N_E\oplus\cO_E))=\int_{(\pi_E\circ\pi)_*\beta_1}c_1(X)|_C+3|\eta|.
\end{eqnarray*}
Hence the dimension constraint becomes
\begin{eqnarray*}
\frac{1}{2}\sum\limits_{i=1}^{\ell(\eta)}deg\delta^{j_i}+\int_{(\pi_E\circ\pi)_*\beta_1}c_1(X)|_C+2|\eta|=\ell(\eta).\nonumber
\end{eqnarray*}
So the dimension constraint holds only if 
\begin{eqnarray*}
\eta=\emptyset,\quad\int_{(\pi_E\circ\pi)_*\beta_1}c_1(X)|_C=0,
\end{eqnarray*}
which implies Lemma \ref{curvelemma2}.
\end{proof}

The above two comparison results give Theorem \ref{curve1}.

To prove Theorem \ref{curve2}, we need the following two lemmas.

\begin{lemma}\label{curvelemma3}
Under the same assumptions as in Theorem \ref{curve2}, we have
\ban
&&Z_P\bigg(X;q|\tau_0([C])\prod\limits_{i=1}^m\tau_{d_i}(\gamma_i)\bigg)_\beta\\
&=&Z_P\bigg(\bP_C(N_C\oplus\cO_C)/D_\infty;q|\tau_0([C])|(1,[pt])\bigg)_F\cdot\frac{1}{q}\cdot Z_P\bigg(\tilde X/E;q|\prod\limits_{i=1}^m\tau_{d_i}(p^*\gamma_i)|(1,\mathbbm 1)\bigg)_{p^!\beta-e},
\nan
where $D_\infty=\bP_C(N_C\oplus\{0\})$, and $F$ is the class of a line in the fiber of $\bP_C(N_C\oplus\cO_C)$.
\end{lemma}
\begin{proof}
Degenerate $X$ along $C$, and by the degeneration formula, we have:
\begin{eqnarray}\label{dcurvelemma3}
&&Z_P\bigg(X;q|\tau_0([C])\prod\limits_{i=1}^m\tau_{d_i}(\gamma_i)\bigg)_\beta\\
&=&\sum Z_P\bigg(\bP_C(N_C\oplus\cO_C)/D_\infty;q|\tau_0[C]|\eta\bigg)_{\beta_1}\cdot\frac{(-1)^{|\eta|-\ell(\eta)}\fz(\eta)}{q^{|\eta|}}\nonumber\\
&&\qquad\cdot Z_P\bigg(\tilde X/E;q|\prod\limits_{i=1}^m\tau_{d_i}(p^*\gamma_i)|\eta^\vee\bigg)_{\beta_2}, \nonumber
\end{eqnarray}
where we have assumed that the support of $\gamma_i$ is away from $C$. Recall that we have assumed that 
\ban
v\textrm{dim}_\bC P_n(X,\beta)=\frac{1}{2}\sum\limits_{i=1}^mdeg\gamma_i+\sum\limits_{i=1}^md_i+1-m.
\nan

Suppose that $(\eta=\{(\eta_i,\delta_{j_i})\}_{i=1}^{\ell(\eta)},\beta_1,\beta_2)$ has nonzero contribution in \eqref{dcurvelemma3}. Then
\ban
&&v\textrm{dim}_\bC P_n(\bP_C(N_C\oplus\cO_C)/D_\infty,\beta_1)=\int_{\beta_1}c_1(\bP_C(N_C\oplus\cO_C)),\\
&&v\textrm{dim}_\bC P_n(\tilde X/E,\beta_2)=\frac{1}{2}\sum_{i=1}^m deg\gamma_i+\sum_{i=1}^md_i+\frac{1}{2}\sum\limits_{i=1}^{\ell(\eta)}deg\delta^{j_i}-\ell(\eta)+|\eta|-m.
\nan
So by the dimension constraint,
\begin{eqnarray*}
\frac{1}{2}\sum\limits_{i=1}^{\ell(\eta)}deg\delta^{j_i}+\int_{\beta_1}c_1(\bP_C(N_C\oplus\cO_C))-|\eta|=1+\ell(\eta).\nonumber
\end{eqnarray*}
As in the proof of Lemma \ref{curvelemma1}, one can check that 
\ban
\int_{\beta_1}c_1(\bP_C(N_C\oplus\cO_C))=\int_{\pi_*\beta_1}c_1(X)|_C+3|\eta|,
\nan
and the dimension constraint becomes
\begin{eqnarray*}
&&\frac{1}{2}\sum\limits_{i=1}^{\ell(\eta)}deg\delta^{j_i}+\int_{\pi_*\beta_1}c_1(X)|_C+2|\eta|=1+\ell(\eta).
\end{eqnarray*}
The dimension constraint holds only if 
\begin{eqnarray*}
\eta=(1,[pt]),\quad\int_{\pi_*\beta_1}c_1(X)|_C=0,
\end{eqnarray*}
which implies Lemma \ref{curvelemma3}.
\end{proof}

\begin{lemma}\label{curvelemma4}
Under the same assumptions as in Theorem \ref{curve2}, we have
\ban
&&Z_P\bigg(\tilde X;q|\prod\limits_{i=1}^m\tau_{d_i}(p^*\gamma_i)\bigg)_{p^!\beta-e}\\
&=&Z_P\bigg(\bP_E(N_E\oplus\cO_E)/D_\infty;q|\tau_0(E)|(1,[pt])\bigg)_F\cdot\frac{1}{q}\cdot Z_P\bigg(\tilde X/E;q|\prod\limits_{i=1}^m\tau_{d_i}(p^*\gamma_i)|(1,\mathbbm 1)\bigg)_{p^!\beta-e},
\nan
where $D_\infty=\bP_E(N_E\oplus\{0\})$, and $F$ the class of a line in the fiber of $\bP_E(N_E\oplus\cO_E)$.
\end{lemma}
\begin{proof}
Degenerate $\tilde X$ along $E$, and by the degeneration formula, we have
\begin{eqnarray}\label{dcurvelemma4}
&&Z_P\bigg(\tilde X;q|\tau_0(E)\prod\limits_{i=1}^m\tau_{d_i}(p^*\gamma_i)\bigg)_{p^!\beta-e}\\
&=&\sum Z_P\bigg(\bP_E(N_E\oplus\cO_E)/D_\infty;q||\eta\bigg)_{\beta_1}\cdot\frac{(-1)^{|\eta|-\ell(\eta)}\fz(\eta)}{q^{|\eta|}}\nonumber\\
&&\qquad\cdot Z_P\bigg(\tilde X/E;q|\prod\limits_{i=1}^m\tau_{d_i}(p^*\gamma_i)|\eta^\vee\bigg)_{\beta_2},\nonumber
\end{eqnarray}
where we have assumed that the support of $p^*\gamma_i$ is away from $E$. Recall that we have assumed that
\ban
v\textrm{dim}_\bC P_n(\tilde X,p^!\beta-e)=\frac{1}{2}\sum\limits_{i=1}^mdeg\gamma_i+\sum\limits_{i=1}^md_i-m.
\nan

Assume that $(\eta=\{(\eta_i,\delta_{j_i})\}_{i=1}^{\ell(\eta)},\beta_1,\beta_2)$ has nontrivial contribution in \eqref{dcurvelemma4}, and then
\begin{eqnarray*}
&&v\textrm{dim}_\bC P_n(\bP_E(N_E\oplus\cO_E)/D_\infty,\beta_1)=\int_{\beta_1}c_1(\bP_E(N_E\oplus\cO_E)),\\
&&v\textrm{dim}_\bC P_n(\tilde X/E,\beta_2)=\frac{1}{2}\sum_{i=1}^m deg\gamma_i+\sum_{i=1}^md_i+\frac{1}{2}\sum\limits_{i=1}^{\ell(\eta)}deg\delta^{j_i}-\ell(\eta)+|\eta|-m.
\end{eqnarray*}
So by the dimension constraint,
\begin{eqnarray*}
\frac{1}{2}\sum\limits_{i=1}^{\ell(\eta)}\textrm{deg}\delta^{j_i}+\int_{\beta_1}c_1(\bP_E(N_E\oplus\cO_E))-|\eta|=\ell(\eta).\nonumber
\end{eqnarray*}
Following the notations in the proof of Theorem \ref{curve0}, we have the following constraints for $\beta_1$:
\begin{eqnarray*}
\left\{\begin{array}{ccl}\beta_1\cdot D_\infty&=&|\eta|,\\\beta_1\cdot E&=&1,\end{array}\right.
\end{eqnarray*}
and this gives 
\begin{eqnarray*}
\pi_*\beta_1\cdot E=-|\eta|+1,
\end{eqnarray*}
which implies that
\begin{eqnarray*}
\int_{\beta_1}c_1(\bP_E(N_E\oplus\cO_E))=\int_{(\pi_E\circ\pi)_*\beta_1}c_1(X)|_C+3|\eta|-1.
\end{eqnarray*}
Hence the dimension constraint becomes
\begin{eqnarray*}
\frac{1}{2}\sum\limits_{i=1}^{\ell(\eta)}deg\delta^{j_i}+\int_{(\pi_E\circ\pi)_*\beta_1}c_1(X)|_C+2|\eta|=1+\ell(\eta).\nonumber
\end{eqnarray*}
So the dimension constraint holds only if 
\begin{eqnarray*}
\eta=(1,[pt]),\quad\int_{(\pi_E\circ\pi)_*\beta_1}c_1(X)|_C=0,
\end{eqnarray*}
which implies Lemma \ref{curvelemma4}.
\end{proof}

{\bf Proof of Theorem \ref{curve2}:} By Lemma \ref{curvelemma3} and \ref{curvelemma4}, in the particular case $X=\bP_C(N_C\oplus\cO_C)$, we have
\ban
\frac{Z_P\bigg(\bP_C(N_C\oplus\cO_C);q|\tau_0([C])\tau_0([pt])\bigg)_F}{Z_P\bigg(\bP_E(N_E\oplus\cO_E);q|\tau_0(E)\tau_0([pt])\bigg)_{F}}=\frac{Z_P\bigg(\bP_C(N_C\oplus\cO_C)/D_\infty;q|\tau_0([C])|(1,[pt])\bigg)_F}{Z_P\bigg(\bP_E(N_E\oplus\cO_E)/D_\infty;q|\tau_0(E)|(1,[pt])\bigg)_F}.
\nan
By (4.2) in \cite{PT}, we have
\ban
Z_P\bigg(\bP_C(N_C\oplus\cO_C)/D_\infty;q|\tau_0([C])|(1,[pt])\bigg)_F&=&q(1+q)^,\\
Z_P\bigg(\bP_E(N_E\oplus\cO_E)/D_\infty;q|\tau_0(E)|(1,[pt])\bigg)_F&=&q.
\nan
which gives Theorem \ref{curve2}.

{\bf Acknowledgements.}
The author would like to thank Jianxun Hu for valuable comments on earlier versions of this paper, and Yongbin Ruan, Wei-Ping Li and Zhenbo Qin for many helpful discussions.

\end{document}